\documentclass{amsart}
\usepackage{amsmath}
\usepackage{amsfonts}
\usepackage{amssymb,amsthm}
\usepackage{color}
\usepackage[margin=1in]{geometry}
\usepackage{graphicx}
\usepackage{verbatim}
\usepackage{hyperref}
\hypersetup{
    colorlinks,
    citecolor=blue,
    filecolor=black,
    linkcolor=black,
    urlcolor=black
}

\usepackage{todonotes}

\numberwithin{equation}{section}
\theoremstyle{definition}
\newtheorem{theorem}{Theorem}[section]
\newtheorem{proposition}[theorem]{Proposition}
\newtheorem{lemma}[theorem]{Lemma}

\newtheorem*{remark}{Remark}
\newtheorem{definition}[equation]{Definition}

\newcommand{\Res}{\operatorname{Res}}

\newcommand{\ZZ}{\mathbb{Z}}

\newcommand{\h}{\mathfrak{h}}

\newcommand{\HH}{\mathcal{H}}

\newcommand{\Sym}{\operatorname{Sym}}
\renewcommand{\char}{\text{char}}
\newcommand{\sspan}{\text{span}}

\newcommand{\PP}{\mathbb{P}}

\begin{document}
\title[Polynomial representation of type $A_{n - 1}$ rational Cherednik algebra in characteristic $p \mid n$]{The polynomial representation of the type $A_{n - 1}$ rational Cherednik algebra in characteristic $p \mid n$}
\author{Sheela Devadas}
\email[S. Devadas]{sheelad@stanford.edu}
\author{Yi Sun}
\email[Y. Sun]{yisun@math.mit.edu}
\date{\today}

\begin{abstract}
We study the polynomial representation of the rational Cherednik algebra of type $A_{n-1}$ with generic parameter in characteristic $p$ for $p \mid n$. We give explicit formulas for generators for the maximal proper graded submodule, show that they cut out a complete intersection, and thus compute the Hilbert series of the irreducible quotient. Our methods are motivated by taking characteristic $p$ analogues of existing characteristic $0$ results.
\end{abstract}

\maketitle
\setcounter{tocdepth}{1}
\tableofcontents

\section{Introduction}

The present work presents a detailed study of the polynomial representation of the type $A_{n - 1}$ rational Cherednik algebra over a field of characteristic $p$ dividing $n$.  Rational Cherednik algebras were introduced by Etingof-Ginzburg in \cite{EG} as a rational degeneration of the double affine Hecke algebra dependent on two parameters $\hbar$ and $c$.  In characteristic $0$, their type $A$ representation theory has been the subject of extensive study.  We refer the reader to \cite{EM} for a survey of these results. 

In characteristic $p$ and especially in the modular case, much less is known about the representation theory of the rational Cherednik algebra.  In this paper, we consider the modular case $p \mid n$.  For $\hbar = 1$ and generic $c$, we provide a complete characterization of the irreducible quotient of the polynomial representation.  We give explicit generators for the unique maximal proper graded submodule $J_c$, show that the irreducible quotient is a complete intersection, and compute its Hilbert series.

Our techniques are inspired by taking characteristic $p$ analogues of results about Cherednik algebras in characteristic $0$.  In particular, our explicit expression for generators of $J_c$ was obtained by converting expressions with complex residues to equivalent expressions dealing only with formal power series which may be interpreted in characteristic $p$.  While we restrict our study to the polynomial representation in type $A$, we view it as a test case for this philosophy, which we believe may admit wider application.

We now state our results precisely and explain their relation to other recent work.

\subsection{The rational Cherednik algebra in positive characteristic}

We work over an algebraically closed field $k$ of characteristic $p > 0$ and fix $n$ so that $p \mid n$.  Let $S_n$ denote the symmetric group on $n$ elements, $V = k^n$ its permutation representation, and $s_{ij} \in S_n$ the transposition permuting $i$ and $j$.  Fix a basis $y_1,\ldots,y_{n}$ for $V$ and a dual basis $x_1, \ldots, x_{n}$ for $V^*$.  Let $\h$ and $\h^*$ be the dual $(n - 1)$-dimensional $S_n$-representations which are subrepresentation and quotient of $V$ and $V^*$, respectively given by
\[
\h = \sspan\{y_i - y_j \mid i \neq j\} \text{ and } \h^* = V^*/(x_1 + \cdots + x_{n}).
\]
The action of $S_n$ on $\h$ and $\h^*$ is given explicitly by natural permutation of basis vectors.

Fix constants $\hbar$ and $c$ in $k$.  Denoting the tensor algebra of $\h \oplus \h^*$ by $T(\h \oplus \h^*)$, the \textit{type $A_{n - 1}$ rational Cherednik algebra} $\HH_{\hbar, c}(\h)$ is the quotient of $k[S_n] \ltimes T(\h \oplus \h^*)$ by the relations
\begin{align*}
[x_i,x_j]=0, \quad [y_i-y_j,y_l-y_m] = 0, \quad [y_i-y_j,x_i] = \hbar- cs_{ij}-c \sum_{t \ne i} s_{it},\quad [y_i-y_j,x_l] =cs_{il}-cs_{jl}\end{align*}
for all $1 \le i , j, l, m \le n$ such that $i,j,l$ are distinct and $l \neq m$. There is a $\ZZ$-grading on $\HH_{\hbar,c}(\h)$ given by setting $\deg x=1$ for $x \in \h^*$, $\deg y = -1$ for $y \in \h$, and $\deg g=0$ for $g \in k[S_n]$.  In addition, $\HH_{\hbar, c}(\h)$ admits a PBW decomposition 
\[
\HH_{\hbar,c}(\h) \simeq \Sym(\h) \otimes_k k[S_n] \otimes_k \Sym(\h^*).
\]
For any $\alpha \ne 0$, $\HH_{\hbar,c}(\h)$ and $\HH_{\alpha\hbar,\alpha c}(\h)$ are isomorphic as algebras, so only the cases $\hbar = 0$ or $\hbar = 1$ need be considered.  In this paper, we restrict our attention to $\hbar = 1$.

\subsection{Polynomial representation of the rational Cherednik algebra}

The rational Cherednik algebra $\HH_{1, c}(\h)$ admits a $\ZZ_{\geq 0}$-graded representation on the polynomial ring $A = \Sym(\h^*)$, known as the \textit{polynomial representation}.  The actions of $\Sym(\h^*)$ and $k[S_n]$ on $A$ are by left multiplication and the $S_n$ action on $\h^*$, respectively.  The action of $\Sym(\h)$ is implemented by letting $y \in \h$ act by the \textit{Dunkl operator}
\begin{align*}
D_y =  \partial_y  - \sum_{m < l} c  \langle y, x_m-x_l \rangle \frac{1-s_{ml}}{x_m - x_l},
\end{align*}
where we note that $\frac{1-s_{ml}}{x_m - x_l}f$ is a polynomial for $f \in A$. Explicitly, we have  
\[
D_{y_i - y_j} =  \partial_{y_i-y_j}-c\sum_{m \ne i} \frac{1-s_{mi}}{x_i-x_m}+c\sum_{m \ne j} \frac{1-s_{mj}}{x_j-x_m},
\] \noindent
where $\partial_{y_i-y_j}$ is the differential operator satisfying $\partial_{y_i-y_j}(x) = \langle y_i-y_j,x\rangle$ for all $x \in \h^*$.  

\subsection{Maximal proper graded submodule and irreducible quotient of polynomial representation}

As described in \cite[Section 2.5]{BC1}, there is a contravariant form 
\[
\beta_c: \Sym(\h^*) \otimes \Sym(\h) \to k
\]
defined by setting $\beta_c(1, 1) = 1$ and imposing for all $x \in \h^*, y \in \h, f \in \Sym(\h^*), g \in \Sym(\h)$ that
\[
\beta_c(xf,g)=\beta_c(f,D_x(g)) \qquad \text{ and } \qquad \beta_c(f,yg) = \beta_c(D_y(f),g).
\]
where for $x \in \h^*$ we denote by $D_x$ the Dunkl operator implementing the action of $\HH_{1, c}(\h^*)$ on its polynomial representation $\Sym(\h)$. 

The polynomial representation $\Sym(\h^*)$ has unique maximal graded proper submodule $J_c = \ker(\beta_c)$.  By the definition of $\beta_c$, $J_c$ contains the ideal generated by all homogeneous vectors $f \in A$ of positive degree that lie in the kernel of all Dunkl operators $D_y$.  Such $f$ are known as \textit{singular vectors}.  The quotient $L = A/J_c$ is an irreducible representation of $\HH_{1,c}(\h)$.  It inherits a $\ZZ_{\geq 0}$-grading from $A$, and for $L_j$ the degree $j$ subspace of $L$, we may define its Hilbert series as
\[
h_L(t) = \sum_{j \geq 0} \dim L_j t^j.
\]

\subsection{Statement of the main result}

For a formal power series $r(z)$, we denote by $[z^l] r(z)$ the coefficient of $z^l$ in $r(z)$.  Throughout the paper, we will consider formal power series in $z$ considered as expansions of rational functions around $z = 0$. For $i = 1, \ldots, n - 1$, define the formal power series
\[
F_i(z)=\frac{1}{1-x_iz} \sum_{m=0}^{p-1} \binom{c}{m}\left(\prod_{j=1}^{n} (1-x_jz) - 1\right)^m
\]
for $\binom{c}{m} = \frac{c (c - 1) \cdots (c - m + 1)}{m!}$.  Denote by $f_i$ the coefficients $f_i = [z^p] F_i(z)$.

\newtheorem*{thm:main}{Theorem \ref{thm:main}} \begin{thm:main}
For generic $c$, $f_1, \ldots, f_{n-1}$ are linearly independent and generate the maximal proper graded submodule $J_c$ of the polynomial representation for $\HH_{1, c}(\h)$.  The irreducible quotient $L = A/J_c$ is a complete intersection with Hilbert series 
\[
h_L(t) = \left(\frac{1-t^p}{1-t}\right)^{n-1}.
\]
\end{thm:main}
\begin{remark}
In Theorem \ref{thm:main}, by generic $c$ we mean $c$ avoiding finitely many values.
\end{remark}

\subsection{Connections to previous work}

Our study is motivated by previous work on the representation theory of the type $A$ rational Cherednik algebra in both characteristic $0$ and $p$.  The type $A$ non-modular case $p \gg n$ was studied in \cite{BFG}, and some properties of the maximal proper graded submodule of the polynomial representation were given in both modular and non-modular cases in \cite{BC1}.  In the modular case $p \mid n$, for $p = 2$ the polynomial representation associated to the $n$-dimensional permutation representation was studied in \cite{L}.
\begin{theorem}[{\cite[Theorem 5.1]{L}}] \label{thm:lian}
The irreducible quotient of the polynomial representation associated to the $n$-dimensional permutation representation is a complete intersection with Hilbert series
\[
h(t) = (1 + t)^n (1 + t^2).
\]
The corresponding maximal proper graded submodule is generated by $n - 1$ elements of degree $2$ and one element of degree $4$. 
\end{theorem} 
It was further conjectured by Lian in \cite[Conjecture 5.2]{L} that for all $p$ the corresponding irreducible is a complete intersection with $J_c$ having $n - 1$ generators in degree $p$ and a single generator in degree $p^2$.  Our results are consistent with the restriction of Lian's conjecture to the case when $\h$ is the $(n - 1)$-dimensional quotient.  It would be interesting to extend our work to prove Lian's conjecture in full.  For general $p \mid n$, a submodule of the maximal proper graded submodule was computed in \cite[Proposition 6.1]{DS}.

In characteristic $0$, our results parallel the explicit decomposition of the polynomial representation of the type $A$ rational Cherednik algebra given in \cite{BEG, CE}.  There, the polynomial representation is irreducible unless $c = \frac{r}{n}$ for some integer $r$, and an explicit set of generators of the maximal proper graded submodule is known.
\begin{proposition}[{\cite[Proposition 3.1]{CE}}] \label{prop:ce}
If $\char(k) = 0$ and $c = \frac{r}{n}$, the maximal proper graded submodule $J_c \subset A$ of the polynomial representation $A$ of $\HH_{1,c}(\h)$ is generated by
\[
\Res_\infty\left[\frac{dz}{z-x_j} \prod_{i=1}^{n} (z-x_i)^c\right] \text{ for $j=1,\dots,n-1$}.
\]
\end{proposition}
We interpret the characteristic $p$ analogue of Proposition \ref{prop:ce} to mean that if $r = p$ and $p \mid n$, then since $p/n$ is equivalent to $0/0$ and thus an indeterminacy in characteristic $p$, taking $c=p/n$ in characteristic $0$ should correspond to taking $c$ generic in characteristic $p$.  While this substitution is of course invalid, Proposition \ref{prop:ce} may be interpreted as a statement about certain formal power series.  By using a power series version of this construction of generators which makes sense in characteristic $p$, we are able to mimic the arguments of \cite{BEG, CE} to show that they cut out a complete intersection and generate the entire ideal.  We believe that the philosophy of taking characteristic $p$ analogues of characteristic $0$ results for the rational Cherednik algebra should apply more generally and hope to explore this further in future work.

\subsection{Outline of the paper}

The remainder of this paper is organized as follows.  In Section 2, we check that the generators $f_1, \ldots, f_{n - 1}$ are linearly independent singular vectors.  In Section 3, we show that they cut out a complete intersection.  In Section 4, we put these facts together to conclude Theorem \ref{thm:main}.

\subsection{Acknowledgements} 

The authors thank P. Etingof for suggesting the problem and for helpful discussions; the authors thank also the anonymous referee for comments leading to the streamlined proof of Proposition \ref{prop:ci} which appears here.  Some exploratory computations were done using Sage.  S.~D. was supported by the MIT Undergraduate Research Opportunities Program (UROP). Y.~S. was supported by a NSF Graduate Research Fellowship (NSF Grant \#1122374).  Both authors were also supported by NSF Grant DMS-1000113.

\section{An explicit construction of singular vectors}

\subsection{Definition of the singular vectors}

In $A$, define the polynomials 
\[
g(z)=\prod_{j=1}^{n} (1-x_jz) \qquad \text{ and } \qquad F(z) = \sum_{m=0}^{p-1} \binom{c}{m} (g(z)-1)^m.
\]
In these terms, we have $F_i(z) = \frac{F(z)}{1-x_iz}$ and $f_i = [z^p] \frac{F(z)}{1-x_iz}$.  We will show that $f_i$ are singular vectors. 

\subsection{Computation of some partial derivatives}

We begin by computing some partial derivatives of $F$ which will be useful for computing the action of the Dunkl operators.

\begin{lemma}\label{lem:z2g}
We have $[z^0]g(z)=1$ and $[z^1]g(z)=0$, meaning $z^2 \mid g(z)-1$.
\end{lemma}
\begin{proof}
For elementary symmetric polynomials $e_2, \ldots, e_n$, we have the expansion 
\[
g(z)=\prod_{j=1}^{n} (1-x_jz) = 1-z\sum_i x_i+z^2e_2(x_1, \ldots, x_{n})+\dots+(-1)^nz^ne_n(x_1, \ldots, x_{n}).
\]
Recalling that $\sum_i x_i=0$ in $A$, we see that $[z^1]g(z)=0$ and $[z^0]g(z)=1$, so $z^2 \mid g(z)-1$ as desired.
\end{proof}

\begin{lemma}\label{lem:dFdz}
For some formal power series $V(z)$ with $[z^l]V(z)=0$ for $l=0,\dots,p-1$, we have
\[
F'(z) = V(z) - \sum_{j=1}^{n} \frac{cx_j}{1-x_jz} F(z).
\]
\end{lemma}

\begin{proof}
We see easily that $\frac{\partial g}{\partial z} = -g(z) \sum_j \frac{x_j}{1-x_jz}$. We now consider $\frac{\partial F}{\partial z}$. We compute
\begin{align*}
\frac{\partial F}{\partial z}
&=\sum_{m=1}^{p-1}m\binom{c}{m}(g(z)-1)^{m-1}\frac{\partial g}{\partial z}\\
&=-\sum_j \frac{x_j}{1-x_jz}\sum_{m=0}^{p-2}c\binom{c-1}{m}(g(z)-1)^{m}(g(z)-1+1)\\
&=-\sum_j \frac{x_j}{1-x_jz}\left(\sum_{m=0}^{p-2}c\binom{c-1}{m}(g(z)-1)^{m}+\sum_{m=1}^{p-1}c\binom{c-1}{m-1}(g(z)-1)^{m}\right)\\
&=-\sum_j \frac{x_j}{1-x_jz}\left(\sum_{m=0}^{p-1}c\binom{c}{m}(g(z)-1)^{m}-c\binom{c-1}{p-1}(g(z)-1)^{p-1}\right)\\
&=-\sum_j \frac{cx_j}{1-x_jz}F(z)+\sum_j \frac{x_j}{1-x_jz}c\binom{c-1}{p-1}(g(z)-1)^{p-1}.
\end{align*}
Defining the formal power series
\[
V(z)=\sum_j \frac{x_j}{1-x_jz}c\binom{c-1}{p-1}(g(z)-1)^{p-1},
\]
we see that $F'(z) = V(z) - \sum_{j=1}^{n} \frac{cx_j}{1-x_jz} F(z)$. It remains only to show that $[z^l]V(z)=0$ for $l=0,\dots,p-1$, which follows by noting that $(g(z)-1)^{p-1} \mid V(z)$, applying Lemma~\ref{lem:z2g}, and noting $p \ge 2$.
\end{proof}

\begin{lemma}\label{lem:dFdxi} 
For some formal power series $G(z)$ with $[z^l]G(z)=0$ for $l=0,\dots,p$, we have
\[
\partial_{y_2-y_1}(F(z))= G(z) - \left(\frac{zc}{1-x_2z}-\frac{zc}{1-x_1z}\right)F(z).
\]
\end{lemma}
\begin{proof}
We may compute $\partial_{y_2-y_1}(g(z))=g(z)\left(-\frac{z}{1-x_2z}+\frac{z}{1-x_1z}\right)$. Using this, we see that
\begin{align*}
\partial_{y_2-y_1}(F(z))&=\left(\sum_{m=1}^{p-1}m\binom{c}{m}(g(z)-1)^{m-1}\right)\partial_{y_2-y_1}(g(z))\\
&=\left(-\frac{z}{1-x_2z}+\frac{z}{1-x_1z}\right)\left(\sum_{m=1}^{p-1}m\binom{c}{m}(g(z)-1)^{m-1}\right)g(z)\\
&=\left(-\frac{z}{1-x_2z}+\frac{z}{1-x_1z}\right)\left(\sum_{m=0}^{p-2}c\binom{c-1}{m}(g(z)-1)^{m}+\sum_{m=0}^{p-2}c\binom{c-1}{m}(g(z)-1)^{m+1}\right)\\
&=\left(-\frac{zc}{1-x_2z}+\frac{zc}{1-x_1z}\right)\left(F(z)-\binom{c-1}{p-1}(g(z)-1)^{p-1}\right).
\end{align*}
Defining $G(z)=\left(\frac{zc}{1-x_2z}- \frac{zc}{1-x_1z}\right)\binom{c-1}{p-1}(g(z)-1)^{p-1}$, we have shown that
\[
\partial_{y_2-y_1}(F(z))= G(z) - \left(\frac{zc}{1-x_2z}-\frac{zc}{1-x_1z}\right)F(z)
\]
It remains only to show that $[z^l]G(z)=0$ for $l=0,\dots,p$, which follows by noting that $z(g(z)-1)^{p-1} \mid G(z)$, applying Lemma~\ref{lem:z2g}, and noting $p \geq 2$.
\end{proof}

\subsection{Proving $f_1, \ldots, f_{n-1}$ are singular vectors} 

\begin{proposition}\label{prop:ann} 
The elements $f_1, \ldots, f_{n - 1}$ are singular vectors in $A$.
\end{proposition}
\begin{proof}
We must show that for $i = 1, \ldots, n - 1$, $f_i$ is annihilated by $D_{y_j - y_l}$ for all $j \ne l$.  First, by symmetry it suffices to consider $f_1$.  Because the Dunkl operators $D_{y_i-y_j}$ for all $i \ne j$ are spanned by $D_{y_i-y_1}$ for $1< i \le n$, it suffices to show $D_{y_i - y_1} f_1 = 0$.  Finally, because $f_1$ is symmetric in the $x_i$ for $i > 1$, it suffices to show that $D_{y_2 - y_1} f_1 = 0$.

Recall by Lemma~\ref{lem:dFdxi} that $\partial_{y_2-y_1}(F(z))= G(z) - \left(\frac{zc}{1-x_2z}-\frac{zc}{1-x_1z}\right)F(z)$ for a power series $G(z)$ with $[z^l]G(z)=0$ for $l=0,\dots,p$. In terms of $G(z)$, we can calculate $\partial_{y_2-y_1}(F_1(z))$ as
\begin{align*}
\partial_{y_2-y_1}(F_1(z))
&=-\frac{z}{(1-x_1z)^2}F(z)+\frac{1}{1-x_1z}\partial_{y_2-y_1}(F(z))\\
&=-\frac{z}{1-x_1z}F_1(z)+\frac{1}{1-x_1z}\left(\frac{zc}{1-x_1z}-\frac{zc}{1-x_2z}\right)F(z)+\frac{G(z)}{1-x_1z}\\
&=\left(\frac{z(c-1)}{1-x_1z}-\frac{zc}{1-x_2z}\right)F_1(z)+\frac{G(z)}{1-x_1z}.
\end{align*}
In addition, we have that 
\[
\frac{1-s_{1i}}{x_1-x_i}(F_1(z))
=\frac{1}{x_1-x_i}\left(\frac{1}{1-x_1z}-\frac{1}{1-x_iz}\right)F(z)=\frac{z}{(1-x_iz)(1-x_1z)}F(z)=\frac{z}{1-x_iz}F_1(z).
\]
Finally, by Lemma~\ref{lem:dFdz}, we have
\[
\frac{\partial F}{\partial z}=V(z) - \sum_j \frac{cx_j}{1-x_jz}F(z),
\]
where $[z^l]V(z)=0$ for $l=0,\dots,p-1$. From this, it follows that
\begin{align*}
\frac{\partial F_1}{\partial z} = \frac{\partial}{\partial z}\left(\frac{F(z)}{1-x_1z}\right) &=\frac{1}{1-x_1z} \frac{\partial F}{\partial z}+\frac{x_1}{(1-x_1z)^2} F(z)\\
&=\frac{V(z)}{1-x_1z}-\frac{1}{1-x_1z}\sum_j \frac{cx_j}{1-x_jz}F(z)+\frac{x_1}{(1-x_1z)^2} F(z)\\
&=\frac{V(z)}{1-x_1z} - \sum_j \frac{cx_j}{1-x_jz}F_1(z)+\frac{x_1}{1-x_1z} F_1(z).
\end{align*}
Because $F_1(z)$ is invariant under $s_{ij}$ for $i,j > 1$, we now compute
\begin{align*}
D_{y_2-y_1}(F_1(z))&=\left(\partial_{y_2-y_1}-c\frac{1-s_{12}}{x_2-x_1}+c\sum_{j > 1} \frac{1-s_{1j}}{x_1-x_j}\right)F_1(z)\\
&=\partial_{y_2-y_1}(F_1(z))-c\frac{1-s_{12}}{x_2-x_1}F_1(z)+c\sum_{j > 1} \frac{1-s_{1j}}{x_1-x_j}F_1(z)\\
&=\frac{G(z)}{1-x_1z}+\left(\frac{z(c-1)}{1-x_1z}-\frac{zc}{1-x_2z}\right)F_1(z)+\frac{zc}{1-x_2z}F_1(z)+\sum_{j > 1} \frac{zc}{1-x_jz}F_1(z)\\
&=\frac{G(z)}{1-x_1z}-\frac{z}{1-x_1z}F_1(z)+\sum_{j} \frac{zc}{1-x_jz}F_1(z)\\
&=\frac{G(z)}{1-x_1z}-zF_1(z)+zF_1(z)-\frac{z}{1-x_1z}F_1(z) + \sum_{j} \left(\frac{zc}{1-x_jz} - zc\right) F_1(z)\\
&=\frac{G(z)}{1-x_1z}-zF_1(z)-\frac{x_1z^2}{1-x_1z}F_1(z)+\sum_{j} \frac{x_jcz^2}{1-x_jz}F_1(z)\\
&=\frac{G(z)}{1-x_1z}-zF_1(z)-z^2\frac{\partial F_1}{\partial z}+z^2\frac{V(z)}{1-x_1z}\\
&=\frac{G(z)+z^2V(z)}{1-x_1z}-zF_1(z)-z^2\frac{\partial F_1}{\partial z},
\end{align*}
where in the fifth step we have subtracted $n zc F_1(z)$.  We note that $[z^p]\,\frac{G(z)+z^2V(z)}{1-x_1z}$ is a linear combination of $[z^l](G(z)+z^2V(z))$ for $0 \le l \le p$, hence a linear combination of $[z^l]G(z)$ for $0 \le l \le p$ and $[z^l]V(z)$ for $0 \le l \le p-2$. By Lemmas~\ref{lem:dFdz} and \ref{lem:dFdxi}, these coefficients of $G(z)$ and $V(z)$ are all $0$, hence $[z^p]\,\frac{G(z)+z^2V(z)}{1-x_1z}=0$. We conclude that 
\[
[z^p]D_{y_2-y_1}(F_1(z))=[z^p]\left(-zF_1(z)-z^2\frac{\partial F_1}{\partial z}\right).
\]
If $b=[z^{p-1}](F_1(z))$, then $[z^p](-zF_1(z))=-b$ and $[z^p]\left(-z^2\frac{\partial F_1}{\partial z}\right)=b$, which implies that 
\[
D_{y_2 - y_1} f_1 = [z^p]D_{y_2-y_1}(F_1(z))= [z^p]\left(-zF_1(z)-z^2\frac{\partial F_1}{\partial z}\right) = -b + b = 0. \qedhere
\]
\end{proof}

\subsection{Proof of linear independence of $f_1, \ldots, f_{n-1}$}

\begin{proposition}\label{prop:linind} 
For generic $c$, $f_1, \ldots, f_{n - 1}$ are linearly independent degree $p$ homogeneous polynomials.
\end{proposition} 
\begin{proof} 
We have the expansion
\[
F_i(z)=\frac{1}{1-x_iz}\sum_{m=0}^{p-1} \binom{c}{m} (g(z)-1)^m=\sum_{l=0}^\infty x_i^lz^l\sum_{m=0}^{p-1} \binom{c}{m} (g(z)-1)^m.
\]
Because for any $l$ the coefficient of $z^l$ in each factor is a homogeneous polynomial of degree $l$, we see that $[z^p]F_i(z)$ is homogeneous of degree $p$.

For linear independence, suppose that $\sum_{i=1}^{n-1} \lambda_if_i=0$ for some $\lambda_i \in k$.  Substitute $x_{n} = -1$, $x_j=1$ and $x_i=0$ for $i \ne j,i < n$ so that $g(z)=(1-z)(1+z)=1-z^2$ and hence
\[
F_j(z)=\sum_{l=0}^\infty z^l\sum_{m=0}^{p-1} \binom{c}{m} (-z^2)^m
\]
and 
\[
F_i(z)=\sum_{l=0}^\infty 0^lz^l\sum_{m=0}^{p-1} \binom{c}{m} (-z^2)^m=\sum_{m=0}^{p-1} \binom{c}{m} (-z^2)^m \text{ for $i < n - 1$, $i \neq j$}.
\]
If $p = 2$, we see that $[z^2]F_j(z)=1-c$ and $[z^2]F_i(z)=-c$, so varying $j$ implies that
\[
\lambda_j = c\sum_{i=1}^{n-1} \lambda_i \text{ for all }j.
\]
In particular, all $\lambda_i$ have common value $\lambda \in k$ solving $(1-c(n-1))\lambda = 0$, which for $c \ne -1$ and hence for $c$ generic implies that $\lambda = 0$, giving linear independence. 

If $p > 2$, we have 
\[
[z^p]F_j(z)=f_j=\sum_{m=0}^{(p-1)/2} (-1)^m\binom{c}{m}=\binom{c-1}{(p-1)/2}
\]
and $[z^p]F_i(z)=f_i=0$ for $i < n$ and $i \neq j$. For $c \notin \{1,2,\dots,(p-1)/2\}$ and hence for generic $c$, we have $\binom{c-1}{(p-1)/2} \ne 0$, meaning that 
\[
\sum_{i=1}^{n-1} \lambda_if_i = \lambda_j\binom{c-1}{(p-1)/2}=0,
\]
which implies $\lambda_j=0$.  Varying $j$ implies that $\lambda_j=0$ for all $j$, again yielding linear independence.
\end{proof}

\section{Complete intersection properties}

Consider the homogeneous ideal $I_c = \langle f_1, \ldots, f_{n - 1} \rangle \subset A$ generated by the $f_i$.  In this section, we will show that $A/I_c$ is a complete intersection.  Recall from \cite[Example 11.8]{Har} that for a homogeneous ideal $I \subset k[X_0,\dots,X_m]$ with a minimal set of generators of size $l$, the quotient ring $k[X_0, \ldots, X_m]/I$ is a complete intersection if the closed projective subvariety of $\PP^m$ defined by $I$ has dimension $m - l$.

\begin{proposition}\label{prop:ci}
For generic $c$, the quotient $A/I_c$ is a complete intersection. 
\end{proposition}
\begin{proof}
By Proposition~\ref{prop:linind}, $I_c$ has a set of $n-1$ linearly independent and therefore minimal generators $f_1, \ldots, f_{n-1}$ in degree $p$.  Therefore $A/I_c$ is a complete intersection if and only if $I_c$ cuts out a projective variety of dimension $(n - 2) - (n - 1) = -1$, meaning it is empty.  This occurs if and only if the saturation of $I_c$ is $A$, which occurs if and only if $\dim_k(A/I_c) < \infty$.

If $c = 0$, by definition of $f_i$ we see that $f_i = x_i^p$.  Note that 
\[
f_n = x_n^p = (-x_1 - \cdots - x_{n-1})^p = - x_1^p - \cdots - x_{n - 1}^p = - f_1 - \cdots - f_{n -1} \in I_0,
\]
meaning that $(x_1^p, \ldots, x_n^p) \subset I_0$ and hence $\dim_k(A/I_0) < \infty$.  Now, let $A^d$ and $I_c^d$ denote the degree $d$ pieces of $A$ and $I_c$, respectively.  Choose a monomial basis $\{t_i\}$ for $A^d$ independent of $c$, and let $I_c^d$ be spanned by a finite set of polynomials $\{g_j\}$ with $g_j = \sum_i h_{ji}(c) t_i$.  Notice that $\dim I_c^d$ is given by the size of the maximal non-vanishing minor of the matrix $H = (h_{ji})$.  On the other hand, we just showed that there is some degree $d > 0$ so that $I_0^d = A^d$, meaning that $\dim I_c^d < \dim A^d$ exactly when $c$ lies in the zero set of all size $\dim A^d$ minors of $H$, viewed as polynomials in $c$.  This implies that for all but finitely many $c$ we have $\dim I_c^d = \dim A^d$, hence $\dim_k(A/I_c) < \infty$ and $A/I_c$ is a complete intersection.
\end{proof}

\begin{remark}
Proposition \ref{prop:ci} is a formal power series analogue of \cite[Theorem 3.2]{CE}.  However, our proof differs from the ``residues by parts'' argument which appears there, as the crucial \cite[Lemma 3.2]{CE} fails in the modular case.  It is interesting to note that our proof does not appear to translate to the characteristic $0$ case, as there is no analogue for the operation of specializing $c$ to $0$.
\end{remark}

\section{Proof of the main result}

We now put everything together to obtain our main result.

\begin{theorem}\label{thm:main}
For generic $c$, $f_1, \ldots, f_{n-1}$ are linearly independent and generate the maximal proper graded submodule $J_c$ of the polynomial representation for $\HH_{1, c}(\h)$.  The irreducible quotient $L = A/J_c$ is a complete intersection with Hilbert series 
\[
h_L(t) = \left(\frac{1-t^p}{1-t}\right)^{n-1}.
\]
\end{theorem}
\begin{proof}
By \cite[Proposition 3.4]{BC1}, the Hilbert series of $L$ is 
\[
h_L(t) = \left(\frac{1-t^p}{1-t}\right)^{n-1}h(t^p)
\]
for a polynomial $h(t)$ with nonnegative integer coefficients.  On the other hand, by Propositions~\ref{prop:linind} and \ref{prop:ci}, $A/I_c$ is a complete intersection with $n-1$ linearly independent degree $p$ generators $f_1, \ldots, f_{n-1}$. Its Hilbert series is
\[
h_{A/I_c}(t)=\left(\frac{1-t^p}{1-t}\right)^{n-1}.
\]
By Proposition~\ref{prop:ann}, the generators $f_1, \ldots, f_{n - 1}$ of $I_c$ are singular vectors, so $I_c \subseteq J_c$, implying that $h_{A/I_c}(t) \ge h_{A/J_c}(t)$ coefficient-wise.  We conclude that $h(t) = 1$, hence $h_{A/I_c}(t)=h_{A/J_c}(t)$ and $I_c=J_c$, completing the proof.
\end{proof}

\begin{remark}
In the proof of Proposition \ref{prop:ci}, we require that $c$ avoids $\{-1, 1,\dots,(p-1)/2\}$ for Proposition \ref{prop:linind} and that $c$ avoids a non-explicit finite set given by vanishing of a determinantal ideal.  These are the only uses of the assumption that $c$ is generic, so Proposition \ref{prop:ci} and Theorem \ref{thm:main} hold for $c$ avoiding these values.
\end{remark}

\bibliographystyle{alpha}
\bibliography{rca-bib}

\begin{thebibliography}{BEG03}

\bibitem[BC13]{BC1}
M.~Balagovi{\'c} and H.~Chen.
\newblock Representations of rational {C}herednik algebras in positive
  characteristic.
\newblock {\em J. Pure Appl. Algebra}, 217(4):716--740, 2013.

\bibitem[BEG03]{BEG}
Y.~Berest, P.~Etingof, and V.~Ginzburg.
\newblock Finite-dimensional representations of rational {C}herednik algebras.
\newblock {\em Int. Math. Res. Not.}, (19):1053--1088, 2003.

\bibitem[BFG06]{BFG}
R.~Bezrukavnikov, M.~Finkelberg, and V.~Ginzburg.
\newblock Cherednik algebras and {H}ilbert schemes in characteristic {$p$}.
\newblock {\em Represent. Theory}, 10:254--298, 2006.
\newblock With an appendix by P. Etingof.

\bibitem[CE03]{CE}
T.~Chmutova and P.~Etingof.
\newblock On some representations of the rational {C}herednik algebra.
\newblock {\em Represent. Theory}, 7:641--650, 2003.

\bibitem[DS14]{DS}
S.~Devadas and S.~Sam.
\newblock Representations of rational {C}herednik algebras of {$G(m,r,n)$} in
  positive characteristic.
\newblock {\em J. Commut. Algebra}, 6(4):525--559, 2014.

\bibitem[EG02]{EG}
P.~Etingof and V.~Ginzburg.
\newblock Symplectic reflection algebras, {C}alogero-{M}oser space, and
  deformed {H}arish-{C}handra isomorphism.
\newblock {\em Invent. Math.}, 147(2):243--348, 2002.

\bibitem[EM10]{EM}
P.~Etingof and X.~Ma.
\newblock Lecture notes on {C}herednik algebras, 2010.
\newblock \url{http://arxiv.org/abs/1001.0432v4}.

\bibitem[Har92]{Har}
Joe Harris.
\newblock {\em {Algebraic geometry: a first course}}, volume 133 of {\em
  Graduate texts in mathematics}.
\newblock Springer-Verlag, 1992.

\bibitem[Lia12]{L}
C.~Lian.
\newblock Representations of {C}herednik algebras associated to symmetric and
  dihedral groups in positive characteristic, preprint, 2012.
\newblock \url{http://arxiv.org/abs/1207.0182v1}.

\end{thebibliography}

\end{document}